\documentclass{amsart}

\usepackage{amssymb}

\usepackage{graphicx}
\usepackage{subcaption}
\usepackage{float}

\newtheorem{theorem}{Theorem}[section]
\newtheorem{lemma}[theorem]{Lemma}
\newtheorem{proposition}[theorem]{Proposition}

\theoremstyle{definition}
\newtheorem{definition}[theorem]{Definition}

\theoremstyle{remark}
\newtheorem{remark}[theorem]{Remark}

\numberwithin{equation}{section}

\begin{document}

\title{An Optimal Plank Theorem}


\author{}
\address{}
\curraddr{}
\email{}
\thanks{}

\author{Oscar Ortega-Moreno}
\address{}
\curraddr{}
\email{}
\thanks{The author was supported by the  Austrian  Science  Fund  (FWF),  Projectnumber: P31448-N35, and the Mexican National Council of Science and
Technology (CONACYT), grant number: CVU579817.}


\date{}

\dedicatory{}

\commby{}

\begin{abstract}
We give a new proof of Fejes T\'oth's zone conjecture: for any sequence $v_1,v_2,...,v_n$ of unit vectors in a real Hilbert space $\mathcal{H}$, there exists a unit vector $v$ in $\mathcal{H}$ such that 
\begin{equation*}
	|\langle v_k,v \rangle| \geq \sin(\pi/2n)
\end{equation*}
for all $k$. This can be seen as sharp version of the plank theorem for real Hilbert spaces. Our approach is inspired by Ball's solution to the complex plank problem and thus unifies both the complex and the real solution under the same method.
\end{abstract}

\maketitle

\section{Introduction}

	A plank in a vector space $X$ is the region bounded by two parallel hyperplanes. The classic plank problem, conjectured by Tarski in 1932, states that if an $n$-dimensional convex body is covered by a collection of planks, then the sum of the widths of the planks should be at least the minimal width of the convex body they cover. Tarski proved it for the particular cases of the unit disc and the $3$-dimensional solid sphere. In 1951, Bang \cite{BG51} solved the problem for arbitrary convex bodies. Bang \cite{BG51} also asked whether the widths of the planks could be measured with respect to the convex body that it is covered. Ball \cite{B91} answered affirmatively this affine version of the plank problem for the most interesting case: when the convex body in question is symmetric. Ball's plank theorem can be seen as a generalization of the Hahn--Banach theorem, a sharp quantitative version of the uniform boundedness principle, or a geometric pigeon--hole principle. For planks that are symmetric about the origen, Ball's theorem states the following.

\begin{theorem}[The Plank Theorem]
For any sequence $(\phi_k)_{k = 1}^\infty$ of norm one functionals on a (real) Banach space $X$ and non-negative numbers $(t_k)_{k=1}^\infty$ satisfying 
\begin{equation*}
	\sum_{k=1}^{\infty} t_k < 1,
\end{equation*}  
there exists a unit vector $x$ in $X$ for which 
\begin{equation*}
	|\phi_j(x)|> t_j
\end{equation*}
for every $j$.
\end{theorem}

For an arbitrary Banach space, the condition that the sequence of positive numbers $(t_k)_{k = 1}^\infty$ add up to at most $1$ is sharp. This can be seen by taking the space $X$ to be $\ell_1$ and the collection $\phi_i$ to be the standard basis vectors  in $\ell_\infty$. 
For other spaces, such as  Hilbert spaces, one might expect to be able to improve upon this condition. Ball \cite{B01} proved    that for \emph{complex} Hilbert spaces it is possible to \emph{beat} any sequence for which $\sum_{k} t_k^2 = 1$.

\begin{theorem}[The Complex Plank Theorem \cite{B01}]\label{complexplank}
For any sequence $v_1,v_2,\dots, v_n$ of unit vectors in a complex Hilbert space $\mathcal{H}$ and positive real numbers $t_1,t_2,\dots,t_n$ satisfying  
\begin{equation*}
	\sum_{k=1}^n t^2_k = 1
\end{equation*}
there exists a unit vector $z\in\mathcal{H}$ such that 
\begin{equation*}
	|\langle v_k,z \rangle |\geq t_k
\end{equation*}
for all $k$.
\end{theorem}

On the other hand, for \emph{real} Hilbert spaces, the same statement does not hold because of the following construction: consider $2n$ vectors $v_1,v_2, \dots, v_{2n}$ in $\mathbb{R}^2$ equally spaced around the circle: ($n$ vectors and their negatives). For any unit vector $v$ in $\mathbb{R}^2$ there is a $k$ such that 
\begin{equation*}
	|\langle v_k , v \rangle |\leq \sin(\pi /2n). 
\end{equation*}
	
This simple observation is connected to a conjecture by Fejes T\'oth which was positively answered, about two years ago, by Jiang and Polyanskii in \cite{JP2017}. A zone of width $w$ is the set of points in the unit sphere within spherical distance $w/2$ of a given great circle. In 1973, Fejes T\'oth conjectured that if a collection of zones of equal angular width covers the unit sphere then the angular width of the zones should be at least $\pi/n$.

A zone of spherical width $w$ associated to the great circle normal to a unit vector $v$ is the set given by 
\begin{equation*}
	\{x\in \mathbb{S}^2: |\langle v ,x\rangle|\leq \sin(w/2)\}.
\end{equation*}     

With this notation, Fejes T\'oth's conjecture can be restated and generalized as an optimal plank theorem for real Hilbert spaces. 

\begin{theorem}\label{mainthm}
For any sequence $v_1,v_2,\dots,v_n$ of unit vectors in a real Hilbert space $\mathcal{H}$, there exists a unit vector $v$ in $\mathcal{H}$ such that 
\begin{equation*}
	|\langle v_k,v \rangle| \geq \sin(\pi/2n)
\end{equation*}
for all $k$.  
\end{theorem}

Jiang and Polyanskii \cite{JP2017} solved this conjecture for arbitrary collection of zones (not necessarily having all the same width). They used the classic machinery to solve plank problems: a discrete optimization as in the proof of Bang's lemma \cite{BG51} followed by an additional innovative inductive argument. The purpose of this paper is to give a new proof of Fejes T\'oth's zone conjecture \cite{FT1973} using a completely different method: inverse eigenvectors.

Inverse eigenvectors allow us to transfer a pure geometrical problem to the study of the extremal behaviour of certain class of functions. In the case of the complex plank problem, these are complex polynomials. In the case of Fejes T\'oth's conjecture, these are trigonometric polynomials.

As a second application of inverse eigenvectors we also prove the following theorem. 

\begin{theorem}\label{thm2} For any sequence $v_1,v_2,\dots,v_n$ of vectors in a real Hilbert space $\mathcal{H}$,  there exists a unit vector $v$ in $\mathcal{H}$ such that
\begin{equation}\label{ineqthm2}
	\left\Vert\sum_{k = 1}^n x_k v_k\right\Vert_\mathcal{H}^2\leq n(n-1)\sum_{k = 1}^n x_k^2 |\langle v,v_k\rangle|^2
\end{equation} 
for all $x\in \mathbb{R}^n$.
\end{theorem}

Theorem \ref{thm2} may also be regarded as a strengthening of Tarski's classic plank problem for Hilbert spaces: if $x$ is the $j$th vector of the standard basis of $\mathbb{R}^n$, inequality (\ref{ineqthm2}) becomes 
\begin{equation*}
 	1 \leq n(n-1) |\langle v_j ,v \rangle|^2
\end{equation*}
and hence 
\begin{equation*}
	|\langle v_j ,v \rangle|>\frac{1}{n}
\end{equation*}
for all $j$. 

	The basic strategy in the proofs of Theorems \ref{mainthm} and \ref{thm2} is the strategy followed by Ball in the proof the complex plank problem, but there is a fundamental difference. The main ingredient of the proof of Theorem \ref{complexplank} has no analogue in the real case. In his paper \cite{B01}, Ball studies the behaviour of a complex polynomial locally around one and, with the aid of the maximum modulus principle, manages to find a point in the unit disk where the polynomial has large absolute value. In contrast, our proofs rely on the extremal properties of trigonometric polynomials to show the existence of such a point.

\subsection{Organization of the paper} The rest of the paper is organized as follows. In Section 2 we introduce the notion of inverse eigenvector and see that Theorem \ref{mainthm} can be transform into a statement about location of inverse eigenvectors of a given real Gram matrix of unit vectors. In Section 3, we prove the inverse eigenvector version of Theorem \ref{mainthm}. Finally in Section 4, we prove Theorem \ref{thm2} following the same approach.

\section{Inverse Eigenvectors}

	In this section we introduce the notion of inverse eigenvectors. An inverse eigenvector of a matrix $M$ is a vector $x$ satisfying the equation $Mx = x^{-1}$ where $x^{-1}$ is the inverse of $x$ componentwise.

	Inverse eigenvectors arose naturally in the solution of the complex plank problem. In his paper \cite{B01}, Ball transforms the complex plank problem to a problem concerning the location of inverse eigenvectors of a complex Gram matrix. Seven years later, Leung, Li and Rakesh \cite{LLR} reformulated the problem of finding the polarization constant of $\mathbb{R}^n$ in terms of inverse eigenvectors and described the structure of the inverse eigenvectors for real positive symmetric matrices. 

	The term inverse eigenvector for a vector $x$ satisfying $Mx = x^{-1}$ turns up for the first time in \cite{Amb09}, where Ambrus used the methods in \cite{B01} to reformulate the strong polarization problem as a geometric question concerning the  location of inverse eigenvectors and managed to solve the problem for the planar case. The treatment and adaptations of \cite{B01}  and the definition of inverse eigenvectors that are presented here are due to Ambrus in \cite{Amb09} on his work on the strong polarization problem (see also \cite{LLR}). In order to motivate the definition of inverse eigenvector, let us go back to our question.

	For rest of the discussion, we shall work with the following rescaled version of Theorem \ref{mainthm} which will suit our purposes better. We also assume that $n\geq 2$ so as to eliminate from the discussion the trivial case when $n = 1$.

\begin{theorem}\label{mainp}
For any sequence $v_1,v_2,\dots,v_n$ of unit vectors in a real Hilbert space $\mathcal{H}$, there exists a vector $v\in \mathcal{H}$ of norm $\sqrt{n}$ for which 
\begin{equation*}
	|\langle v_k,v \rangle| \geq \sqrt{n}\sin(\pi/2n)
\end{equation*}
for all $k$.
\end{theorem}

	Our problem consists of finding a vector $v$ of norm $\sqrt{n}$ which has large inner product with all the vectors $v_1, v_2,\dots,v_n$. An obvious candidate for this vector $v$ would be one for which $\min_{k}{|\langle v_k,v \rangle|}$ is maximal. However, there seems to be no simple way to either manipulate or obtain useful information from this maximal condition. Instead we choose a unit vector $v$ for which the product $\prod_{i}|\langle v_i,v \rangle|$ is maximal, hoping that each of the factors will be large enough to get the desired inequality. 

	For the product, we can use simple analytic tools to study the points for which it is locally extremal. Luckily, the structure of these local optimisers can be described concisely as the following proposition shows. The following proposition can be found as discussion in the last paragraph of page 2863 in \cite{LLR} or as Proposition 1.16 in \cite{Amb09}. 

\begin{proposition}\label{motivation}
Let $v_1,v_2,\dots,v_n$ be a sequence of vectors in a real Hilbert space $\mathcal{H}$. Suppose that  $v$ is vector of norm $\sqrt{n}$ chosen so as to maximize  
\begin{equation*} 
	\prod_{k = 1}^n|\left\langle v_k,v \right\rangle|.
\end{equation*}
Then, 
\begin{equation}\label{eq1}
	v = \sum_{k = 1}^n\frac{1}{\left\langle v_k,v \right\rangle}v_k
\end{equation} 
\end{proposition}
\begin{proof} Since $v$ is a stationary point, by the method of Lagrange multipliers, the gradients of the objective function and the constraint should be scalar multiples of one another. Hence, there exists a real number $\lambda$ such that 
\begin{equation} \label{lag}
	v = \lambda \sum_{k = 1}^n\frac{v_k}{\left\langle v_k,v \right\rangle} \prod_{k = 1}^n|\left\langle v_k,v \right\rangle|.
\end{equation}  
This gives equation  (\ref{eq1}) up to a constant and taking inner product with $v$ shows that the constant must be 1.
\end{proof}

	Denote by $H$ the Gram matrix associated to the sequence of vectors $(v_k)_{k=1}^n$, that is, $H_{ij} =  \left\langle v_i,v_j \right\rangle$ for all $i,j$, and let $w$ be the vector in $\mathbb{R}^n$ given by  
\begin{equation*}
	w_k = \frac{1}{\left\langle v_k,v \right\rangle}
\end{equation*} 
for all $k$. Then for every $j$, 
\begin{equation*}
	(Hw)_j = \sum_{i = 1}^n h_{ji}w_i = \langle v_j,\sum_{i = 1}^n w _i v_i \rangle = \langle v_j,v \rangle = \frac{1}{w_j}.
\end{equation*}
Therefore, $w$ satisfies the equation $Hw = w^{-1}$. This observation leads us naturally to the following definition.

\begin{definition}[Ambrus \cite{Amb09}]
Let $M$ be an $n\times n$ matrix. We say that $w$ is an inverse eigenvector of $M$ if 
\begin{equation*}
	M w = w^{-1}
\end{equation*}
\end{definition}

	Although Proposition \ref{motivation} was only used as a motivation for the definition of inverse eigenvectors we will use directly the following analogous statement in terms of inverse eigenvectors themselves.

\begin{proposition}\label{motivation2}
Suppose that $M$ is an $n\times n$ symmetric positive matrix. Let $u$ be a maximizer of 
\begin{equation*}
	\prod_{k = 1}^n \left| u_k\right|
\end{equation*}
subject to the constraint
\begin{equation*}
	u^{\top}M u = n.
\end{equation*}
Then $u$ is an inverse eigenvector of $M$.
\end{proposition}

The proof of Proposition \ref{motivation2} is identical to the proof of Proposition \ref{motivation} and so we omit it. We will now show that Theorem \ref{mainp} can be deduced  easily from the following theorem concerning the location of inverse eigenvectors of a real Gram matrix.

\begin{theorem}\label{loceigen}
Let $H$ be an $n\times n$ real Gram matrix of unit vectors. Then, there exists an inverse eigenvector $w$ of $H$ for which 
\begin{equation}\label{eqnloceigen}
	\Vert w \Vert_{\infty}\leq n^{-\frac{1}{2}}\csc(\pi/2n).
\end{equation}
\end{theorem}

\begin{proof}[Theorem \ref{loceigen} implies Theorem \ref{mainp}] Let $v_1,\dots,v_n$ be a sequence of unit vectors in real Hilbert space $\mathcal{H}$ and let $H$ be the $n \times n$ real Gram matrix associated to this sequence. Let $w$ be an inverse eigenvector satisfying (\ref{eqnloceigen}) and set
\begin{equation*}
	v = \sum_{k = 1}^n w_k v_k.
\end{equation*}
Then for every $j$, 
\begin{equation*}
	\langle v_j,v\rangle  =  \sum_{k = 1}^n \langle v_j, v_k\rangle w_j =  (Hw)_j =  w_k^{-1}
\end{equation*}
and 
\begin{equation*}
	\Vert v\Vert_{\mathcal{H}}^2 = \langle v , v \rangle = \sum_{k = 1}^n \langle v, v_k\rangle w_k = n.
\end{equation*}
Finally since $w$ satifies (\ref{eqnloceigen}),
\begin{equation*}
	|\langle v_j,v\rangle| = |w_k|^{-1}\geq \sqrt{n}\sin(\pi/2n).
\end{equation*}
\end{proof}

 The proof of Theorem \ref{loceigen} will be the subject of the next section.

\section{The Proof of Theorem \ref{loceigen}}

	By a simple approximation argument we may assume that $H$ is positive definite (not
merely semi-definite). The proof may be divided in $2$ steps. First, we find a suitable inverse eigenvector $w$ of $H$. Second, we show that $w$ satisfies the desired condition, i.e.
\begin{equation}\label{secondstep}
	\Vert w \Vert_{\infty}\leq n^{-\frac{1}{2}}\csc(\pi/2n).
\end{equation}

	The next lemma deals with the first step of the proof.

\begin{lemma}\label{findivector}
For any $n \times n$ positive real Gram matrix $H$ there exists an inverse eigenvector $w$ of $H$ satisfying the following: for all $c\in \mathbb{R}^n$ such that $\prod_{k=1}^n|c_k| = 1$, 
\begin{equation}\label{opt}
	\sum_{jk} c_jw^{-1}_j(H^{-1})_{jk}w^{-1}_k c_k \geq n.
\end{equation} 
\end{lemma}
\begin{proof} Let $u$ be a maximizer of 
\begin{equation*}
	\prod_{k = 1}^n \left| u_k\right|
\end{equation*}
subject to the constraint
\begin{equation}\label{constraint}
	u^{\top}H^{-1} u = n.
\end{equation}
By Proposition \ref{motivation2}, we see that $u$ is an inverse eigenvector of $H^{-1}$ and thus $w = u^{-1}$ is an inverse eigenvector of $H$. Notice that for all vector $c$ such that  $\prod \left| c_k\right| = 1$, 
\begin{equation*}
	\prod \left| c_k u_k\right| =  \prod \left| u_k\right|,
\end{equation*}
and therefore since $u$ has being selected as to maximize $\prod \left| u_k\right|$ subject to constraint the constraint (\ref{constraint}) we have that
\begin{equation}\label{opt2}
	\sum_{jk} c_ju_j(H^{-1})_{jk} c_ku_k \geq n.
\end{equation} 
Substituting $u = w^{-1}$ in (\ref{opt2}) yields (\ref{opt}).
\end{proof}

	Having found a suitable inverse eigenvector, we move to the second step of the proof: showing that our inverse eigenvector satisfies (\ref{secondstep}). For this purpose, we will use the following fundamental lemma. We denote by $\mathbf{1}$ the vector whose entries are all equal to $1$.

\begin{lemma}\label{lemmaop}
Suppose that $M$ is an $n\times n$ symmetric positive matrix satisfying 
\begin{itemize}
	\item[$\bullet$] $M\mathbf{1} = \mathbf{1} $, and 
	\item[$\bullet$] $c^{\top}M^{-1} c \geq n$ whenever $c$ is a vector such that $\prod \left| c_k\right| = 1$.
\end{itemize}
Then 
\begin{equation*}
 m_{kk}\leq n^{-1}\csc^2(\pi/2n)
\end{equation*}
for all $k$ where $m_{kk}$ is the $k$th entry on the diagonal of $M$.
\end{lemma}

	For the sake of clarity let us first show how Theorem \ref{loceigen} is almost an immediate consequence of Lemma \ref{findivector} and Lemma \ref{lemmaop}.

\begin{proof}[Lemmas \ref{findivector} and \ref{lemmaop} imply Theorem \ref{loceigen}]
Let $w$ be an eigenvector of $H$ satisfying (\ref{opt}). We need to show that 
\begin{equation*}
	\Vert w \Vert_\infty \leq n^{-\frac{1}{2}}\csc(\pi/2n).
\end{equation*}
We define a new matrix $M$ by 
\begin{equation*}
	m_{jk} = w_j H_{jk} w_k 
\end{equation*}
for all $j,k$. 

Observe that $M$ is positive definite and its inverse is given by 
\begin{equation*}
	M^{-1}_{jk} = w^{-1}_j H^{-1}_{jk} w^{-1}_k.
\end{equation*}
Then 
\begin{equation*}
	(M\mathbf{1})_j =  w_j\sum_{k}H_{jk} w_k =  w_j (Hw)_j = 1,
\end{equation*}
where the last identity is guaranteed by the fact that $w$ is an inverse eigenvector of $H$. Thus, 
\begin{equation}\label{c1}
	M\mathbf{1} = \mathbf{1}.
\end{equation}
On the other hand, by (\ref{opt}), for any vector $c$ such that $\prod|c_k| = 1$, 
\begin{equation}\label{c2}
	c^\top M^{-1} c  = \sum_{jk} c_jw^{-1}_jH_{jk}^{-1}w^{-1}_k c_k \geq n
\end{equation}
Hence, in view of (\ref{c1}) and (\ref{c2}), $M$ satisfies both conditions of Lemma \ref{lemmaop}. Thus for all $k$ 
\begin{equation*}
	m_{kk} = |w_{k}|^2 \leq n^{-1}\csc^2(\pi/2n). 
\end{equation*}
\end{proof}

	Now it only remains to establish Lemma \ref{lemmaop}.

\begin{proof}[Proof of Lemma \ref{lemmaop}] Notice that if we let $c = Mb$ then the second condition of lemma \ref{lemmaop} states that if $\prod{|(Mb)_j| } = 1$, 
\begin{equation*}
	b^{\top} M b \geq n.
\end{equation*}
Let us assume, for a contradiction, that one of the diagonal entries is too large. Thus, without loss of generality assume that  
\begin{equation}\label{m11}
	m_{11}>\frac{1}{n\sin^2(\pi/2n)}.
\end{equation}
We will show that there is a vector $b$ such that $\prod{|(Mb)_j| } \geq 1$, but 
\begin{equation*}
	b^{\top} M b < n.
\end{equation*}
Let 
\begin{equation*}
	\mathcal{E} = \{b: b^{\top} M b = n\}.
\end{equation*}
Consider the following vector 
\begin{equation*}
	v_1 = -\sqrt{\alpha} \frac{n e_1 - \mathbf{1}}{\sqrt{nm_{11}-1}}
\end{equation*}
where 
\begin{equation}\label{alpha}
	\alpha = \frac{\cot^2(\pi/2n)}{nm_{kk}-1}.
\end{equation} 
By a simple rearrengment of (\ref{m11}) observe that $\alpha\in(0,1)$. The vector $v_1$ is inside the ellipsoid $\mathcal{E}$. In fact since $M\mathbf{1} = \mathbf{1}$,
\begin{equation*}
	v_1^{\top}Mv_1   =  \alpha \frac{n^2 e_1^{\top}M e_1 - 2 n e_1^{\top}M \mathbf{1} + n}{nm_{kk}-1} =  \alpha n. 
\end{equation*}
We also notice that $v_1$ is orthogonal to $\mathbf{1}$,
\begin{equation*}
		v_1^{\top}\mathbf{1}   = -\sqrt{\alpha} \frac{n e_1^{\top}\mathbf{1}- n}{\sqrt{nm_{11}-1}} = 0.
\end{equation*}
 For each $\theta \in [0,2\pi]$ define 
\begin{equation*}
	v_\theta  =  \cos \theta \mathbf{1} + \sin \theta v_1. 
\end{equation*}
It is easy to see that $v_\theta$ is just a parametrisation of an ellipse \emph{inside} $\mathcal{E}$. In fact for any  $\theta \in [0,2\pi)$,

	\begin{align*}
	v_\theta^\top Mv_\theta & = \cos^2\theta \mathbf{1}^{\top}M\mathbf{1} + 2 \cos \theta \sin \theta v_1^{\top}M\mathbf{1} + \sin^2 \theta v_1^{\top}M v_1\\
 						    & = n\cos^2\theta + \alpha n\sin^2 \theta.
\end{align*}
Thus for all $\theta\in [0,2\pi)\setminus\{0,\pi\}$,
\begin{equation*}
	v_\theta^\top Mv_\theta < n
\end{equation*}
and 
\begin{equation*}
	v_\theta^\top M v_\theta = n
\end{equation*}
if and only if $\theta = 0$ or $\theta = \pi$. Thus, it suffices to show that there exists $\theta\in [0,2\pi)\setminus\{0,\pi\}$ such that 
\begin{equation*}
	\prod_{j = 1}^n |(M v_\theta)_j|\geq 1.
\end{equation*}
Define the trigonometric polynomial $T_{v_1}$ by 
\begin{align*}
	 T_{v_1}(\theta) & =  \prod_{j = 1}^n (M v_\theta)_j\\
					 & =  \prod_{j = 1}^n\left(\cos \theta + (Mv_1)_j \sin\theta \right).
\end{align*}
for all $\theta\in [0,2\pi)$. Notice that since $M\mathbf{1} = \mathbf{1}$,
\begin{align*}
	(Mv_1)_1 & =  -\sqrt{\alpha}\frac{n (Me_1)_1 - (M\mathbf{1})_1}{\sqrt{nm_{11}-1}}\\
			 & =  -\sqrt{\alpha (nm_{11}-1)}\\
			 & =  -\cot(\pi/2n)
\end{align*}
where in the last identity we used (\ref{alpha}). So the first factor of $T_{v_1}$ is equal to $0$ if and only if  
\begin{equation*}
	\cos \theta =\cot(\pi/2n) \sin \theta,
\end{equation*}
which happens if and only if $\theta = \frac{\pi}{2n}$ or $ \theta = \pi + \frac{\pi}{2n}$. Hence, $T_{v_1}$  has a root at $\theta = \frac{\pi}{2n}$ and $\theta = \pi + \frac{\pi}{2n}$. Expanding the product we get
\begin{equation*}
	T_{v_1}(\theta)  =  \cos^{n} \theta + \sum_{j} (Mv_1)_j\cos^{n-1} \theta \sin \theta +\sin^{2}\theta \text{ }\psi(\theta)
\end{equation*}
where $\psi$ is a trigonometric polynomial of degree at most $n-2$. On the other hand, since $M$ is symmetric, $M\mathbf{1} = \mathbf{1}$ and $v_1$ is orthogonal to $\mathbf{1}$, 
\begin{equation*}
	\sum_{j} (Mv_1)_j = \mathbf{1}^\top M v_1 = \mathbf{1}^\top v_1 = 0
\end{equation*}
and therefore, 
\begin{equation}\label{form11}
	T_{v_1}(\theta)  = \cos^{n} \theta + \sin^{2}\theta \text{ }\psi(\theta).
\end{equation}
It is easy to see that $\cos n \theta$ is of the form (\ref{form11}). Thus taking the difference of $T_{v_1}(\theta)$ and $\cos n \theta$ we get  
\begin{align*}
	Q(\theta) & = T_{v_1}(\theta)-\cos n \theta\\
			  & = \sin^2 \theta\text{ }\psi(\theta)
\end{align*}
where $\psi$ is a trigonometric polynomial of degree at most $n-2$. 

Observe that $Q$ has roots at $0$ and $\pi$, where $T_{v_1}$ and $\cos n\theta$ are both $1$, and at $\frac{\pi}{2n}$ and $\pi+\frac{\pi}{2n}$, where both functions are equal to $0$. 
 
	For a contradiction, let us assume that 
\begin{equation*}
	|T_{v_1}(\theta)|<1
\end{equation*}
for all $\theta\in[\frac{\pi}{n},\frac{(n-1)\pi}{n}]\cup [\frac{(n+1)\pi}{n},\frac{(2n-1)\pi}{n}]$. The extrema of $\cos n \theta $ on $[0,2\pi)$ are located at $\theta_k = \frac{k\pi}{n}$ for $k \in \{0,\dots,2n-1\}$ so 
\begin{equation*}
	\text{sgn } Q(\theta_k) = (-1)^{k+1}.
\end{equation*} 
Thus, by the intermediate value theorem, for each $k\in\{1, \dots,n-2\}\cup\{n+1,\dots,2n-2\}$ there is a $\varphi_k\in (\frac{k\pi}{n},\frac{(k+1)\pi}{n})$ such that $Q(\varphi_k) = 0$. This gives us $2n-4$ additional roots of $Q$.

	Hence, $Q$ has at least $2n$ distinct roots on the interval $[0,2\pi)$. However, 
\begin{equation*}
	Q(\theta) = \sin^2(\theta)\psi (\theta),
\end{equation*}
so it could not have more than $2n-2$ distinct roots. Therefore, there exists $\theta \in [\frac{\pi}{n},\frac{(n-1)\pi}{n}]\cup [\frac{(n+1)\pi}{n},\frac{(2n-1)\pi}{n}]$ such that 
\begin{equation*}
	|T_{v_1}(\theta)| = \prod_{j } |(M v_\theta)_j|\geq 1
\end{equation*}
and $v_\theta^\top Mv_\theta<n$ which is a contradiction to the second condition of  Lemma \ref{lemmaop}.
\end{proof}

\section{The Proof of Theorem \ref{thm2}}

For the proof of Theorem \ref{thm2} we proceed in the same way as for the proof of Theorem \ref{mainp}. As we did for Theorem \ref{mainp}, we will be working with a rescaled version of Theorem \ref{thm2}.

\begin{theorem}\label{thm2s} For any sequence $v_1,v_2,\dots,v_n$ of vectors in real Hilbert space $\mathcal{H}$,  there exists a vector $v$ with norm $\sqrt{n}$ such that
\begin{equation}\label{ineqthm2s}
	\left\Vert\sum_{k=1}^n x_k v_k\right\Vert_\mathcal{H}^2\leq (n-1)\sum_{k = 1}^n x_k^2 |\langle v,v_k\rangle|^2
\end{equation} 
for all $x\in \mathbb{R}^n$.
\end{theorem}

	Now we show that Theorem \ref{thm2s} follows from the next theorem concerning location of inverse eigenvectors of a real Gram matrix.

\begin{theorem}\label{mainp2}
Let $H$ be an $n\times n$ real Gram matrix. Then, there exists an inverse eigenvector $w$ of $H$ for which 
\begin{equation}\label{eqnloceigen2}
\sum_{kj} x_kw_k H_{kj} w_j x_j\leq n-1
\end{equation}
for every unit vector $x$.
\end{theorem}

\begin{proof}[Theorem \ref{mainp2} implies Theorem \ref{thm2s}] Notice that inequality (\ref{ineqthm2s}) is equivalent to 
\begin{equation*}
	\left\Vert\sum_{k=1}^n \frac{x_k}{\langle v,v_k\rangle} v_k\right\Vert_\mathcal{H}^2 \leq n-1
\end{equation*}
for all unit vector $x$.

Let $H$ be the $n \times n$ real Gram matrix associated to a sequence of vectors $v_1,\dots, v_n$ in $\mathcal{H}$. Let $w$ be an inverse eigenvector of $H$ satisfying (\ref{eqnloceigen2}) and set
\begin{equation*}
	v = \sum_{k = 1}^n w_k v_k.
\end{equation*}
Then for every $j$, 
\begin{equation*}
	\langle v_j,v\rangle  =  \sum_{k = 1}^n \langle v_j, v_k\rangle w_j =  (Hw)_j =  w_j^{-1}
\end{equation*}
and hence
\begin{equation*}
	\Vert v\Vert_{\mathcal{H}}^2 = \langle v , v \rangle = \sum_{k = 1}^n \langle v, v_k\rangle w_k = n.
\end{equation*}
Since $w$ satisfies (\ref{eqnloceigen2}), for any unit vector $x$
\begin{equation*}
	\left\Vert\sum_{k=1}^n \frac{x_k}{\langle v,v_k\rangle} v_k\right\Vert_\mathcal{H}^2 = \sum_{kj} x_kw_k H_{kj} w_j x_j \leq n-1.
\end{equation*}
\end{proof}

	Similar to the proof of Theorem \ref{loceigen}, the proof of theorem \ref{mainp2} can be divided in two steps. The first step is to find a suitable inverse eigenvector $w$ of $H$. The second step is to show that our choice of inverse eigenvector works. To deal with the first step we use again Lemma \ref{findivector}. For the second step, we will need the following lemma.   

\begin{lemma}\label{lemmal2}
Suppose that $n \geq 2$ and  $M$ is an $n \times n$  symmetric positive matrix satisfying 
\begin{itemize}
	\item[$\bullet$] $M\mathbf{1} = \mathbf{1} $, and 
	\item[$\bullet$] $c^{\top}M^{-1} c \geq n$ whenever $c$ is a vector such that $\prod \left| c_k\right| = 1$. 
\end{itemize}
Then $\Vert M\Vert_2\leq n-1.$
\end{lemma}

For the sake of clarity let us first show how Theorem \ref{mainp2} is almost an immediate consequence of Lemma \ref{findivector} and Lemma \ref{lemmal2}.

\begin{proof}[Lemmas \ref{findivector} and \ref{lemmal2} imply Theorem \ref{mainp2}] Let $w$ be an inverse eigenvector satisfying (\ref{opt}). We need to show that
\begin{equation*}
	\sum_{kj} x_kw_k H_{kj} w_j x_j\leq n-1
\end{equation*}
for every unit vector $x$. Define a new matrix $M$ by 
\begin{equation*}
	m_{jk} = w_j H_{jk} w_k 
\end{equation*}
for all $j,k$. As we did in the previous section (see the proof of Lemmas \ref{findivector} and \ref{lemmaop} imply Theorem \ref{loceigen}), we see that $M$ satisfies both conditions of Lemma \ref{lemmal2} and therefore 
\begin{equation*}
	\sum_{kj} x_kw_k H_{kj} w_j x_j = x^\top M x  \leq \Vert M \Vert_2  \leq n-1.
\end{equation*}
for every unit vector $x$.
\end{proof}

Now it only remains to establish Lemma \ref{lemmal2}.

\begin{remark}
It must be pointed out that the proof of Lemma \ref{lemmal2} follows the same lines as the proof of Ambrus \cite{Amb09} of the strong polarization problem in the planar case: the contribution here is more a refinement of the proof by using derivatives and Bernstein's inequality which potentially could be applied to a variety of classes of functions that satisfy Bernstein-type inequalities. 
\end{remark}

\begin{proof}[Proof of Lemma \ref{lemmal2}] First notice that if we let $c = Mb$ then the second condition of the lemma can be restated as follows: $\prod{|(Mb)_k| } = 1$ implies 
\begin{equation*}
	b^{\top} M b \geq n.
\end{equation*}
Or equivalently, for any $b$ with 
\begin{equation*}
	b^{\top} M b  =n,
\end{equation*}
$\prod{|(Mb)_k| } \leq 1.$
The proof consists of looking at 2-dimensional slices of the ellipsoid defined by 
\begin{equation*}
	\mathcal{E} = \{x:x^\top M x = n\}.
\end{equation*}
So we will ``cut" $\mathcal{E}$ with subspaces of dimension $2$ of $\mathbb{R}^n$ which contain the vector $\mathbf{1}$. Thus, given a vector $v\in \mathcal{E}$ orthogonal to $\mathbf{1}$, we let $H_v$ be the 2 dimensional subspace spanned by $\mathbf{1}$ and $v$,
\begin{equation*}
	H_v = \text{span}\{\mathbf{1}, v\}.
\end{equation*}
We denote by $\mathcal{E}_v$ the ellipse we get by intersecting $\mathcal{E}$ and $H_v$, \begin{equation*}
	\mathcal{E}_v = \mathcal{E}\cap H_v.
\end{equation*}
Notice that we can parameterize the ellipse $\mathcal{E}_v$ as follows: 
\begin{equation*}
	\mathcal{E}_v =\{v_\theta:\theta\in [0,2\pi)\}
\end{equation*}
where
\begin{equation*}
	v_\theta = \cos\theta\mathbf{1} + \sin \theta v.
\end{equation*}
for all $\theta \in [0,2\pi)$. Define the trigonometric polynomial $T_v$ by 
\begin{align*}
	T_v (\theta) & = \prod_{k = 1}^n (Mv_\theta)_k\\
			     & = \prod_{k = 1}^n (\cos\theta + (Mv)_k\sin \theta)
\end{align*} 

	Notice that $T_v(0) = 1$. We now compute the first and second derivatives of $T_v$ at $0$. For any $\theta$ such that $T_v(\theta)$ is not $0$ we have
\begin{equation}\label{fd}
	\frac{T'_v (\theta)}{T_v(\theta)}  = -\sum_{k = 1}^n\frac{\sin \theta - (Mv)_k \cos\theta }{\cos\theta + (Mv)_k\sin \theta} 
\end{equation}
Evaluating equation (\ref{fd}) at $0$ yields 
\begin{equation*}
	T'_v (0) = \sum_{k = 0}^n(Mv)_k = \mathbf{1}^\top M v = \mathbf{1}^\top v = 0. 
\end{equation*}
Taking derivatives on both sides of equation (\ref{fd}) yields    
\begin{equation}\label{sd}
	\frac{T''_v(\theta)T_v(\theta)-(T'_v(\theta))^2}{T_v(\theta)^2} = -\sum_{k = 1}^n\frac{1 + (Mv)^2_k }{(\cos\theta + (Mv)_k\sin \theta)^2} 
\end{equation}
Thus, replacing $T_v(0) =1$ and $T_v'(0) =0$ in equation (\ref{sd}), we get 
\begin{equation*}
	|T_v''(0)| = n + \Vert Mv\Vert^2
\end{equation*}
We are now in a position to apply the following well known inequality for trigonometric polynomials.
\begin{theorem}[Bernstein's Inequality] For any trigonometric polynomial $T$ of degree at most $n$,   
\begin{equation}
	\Vert T'\Vert_{\infty} \leq n\Vert T\Vert_{\infty}
\end{equation}
where $\Vert T\Vert_\infty$ denotes the uniform norm of $T$ on $[0,2\pi]$.
\end{theorem}
Applying Bernstein's inequality twice, we get the following inequality for the second derivative of $T_v$,
\begin{equation}\label{Ber2}
	\Vert T''_v\Vert_\infty\leq n^2\Vert T_v\Vert_\infty.
\end{equation}
Since $v_\theta\in \mathcal{E}$, 
\begin{equation*}
	|T_v(\theta)|  = \prod_{k = 1}^n |(Mv_\theta)_k|\leq 1
\end{equation*}
for all $\theta$ and thus 
\begin{equation*}
	\Vert T_v\Vert_\infty\leq 1
\end{equation*}
for all $v\in \mathcal{E}\cap \mathbf{1}^{\top}$. Hence by inequality (\ref{Ber2}),  
\begin{equation*}
	n + \Vert Mv\Vert^2 =  |T_v''(0)|\leq \Vert T''_v\Vert_\infty\leq n^2
\end{equation*}
for all $v\in \mathcal{E}\cap \mathbf{1}^{\top}$. Therefore, 
\begin{equation}\label{ineq}
	\Vert{Mv}\Vert^2\leq n(n-1)
\end{equation}
for all $v\in \mathcal{E}\cap \mathbf{1}^{\top}$. 
Let $v\in \mathcal{E}$ be an eigenvector orthogonal to $\mathbf{1}$ associated to the possible largest eigenvalue $\lambda$. For this eigenvector $v$ we have that
	\begin{equation*}
		\Vert Mv \Vert^2 = v^\top M^\top Mv = \lambda v^\top M v  = \lambda n 
	\end{equation*}
and hence by (\ref{ineq}), $$\lambda\leq n-1.$$ The norm $ \Vert M \Vert_2$ is the maximum of $1$ and $\lambda$ which, in either case, is less than or equal to $n-1$.
\end{proof}

\section*{Acknowledgments}
The author thanks Professor Keith M. Ball for his guidance throughout the development of this research work and his numerous remarks and suggestions to improve the presentation. I would also like to thank the referee for their careful reading that led to further substantial improvement in the organization of this paper.  
\bibliographystyle{amsplain}

\bibliography{references.bib}

\end{document}